\documentclass[12pt]{article}

\newcommand{\version}{March 7, 2025}

\usepackage{xcolor}
\usepackage{graphicx}
\usepackage{pictex}
\usepackage{mathrsfs}
\usepackage{amssymb}
\usepackage{enumitem}
\setlist{nosep}

\def\hop#1{\relax} 

\title{Equivalent formulation of Thomassen's conjecture
using Tutte paths in claw-free graphs}

\author{Adam Kabela$^{1,2}$ \and Zden\v{e}k Ryj\'a\v{c}ek$^{1,2}$ 
        \and Petr Vr\'ana$^{1,2}$}
      
\date{\version}

\newcounter{mathitem}
  {\begin{list}{{$(\roman{mathitem})$}}{
   \setcounter{mathitem}{0}
   \usecounter{mathitem}
   \setlength{\topsep}{0pt plus 2pt minus 0pt}
   \setlength{\parskip}{0pt plus 2pt minus 0pt}
   \setlength{\partopsep}{0pt plus 2pt minus 0pt}
   \setlength{\parsep}{0pt plus 2pt minus 0pt}
   \setlength{\leftmargin}{35pt}
   \setlength{\itemsep}{0pt plus 2pt minus 0pt}}}
  {\end{list}}

\newcounter{mathites}
  {\begin{list}{{$(\roman{mathites})$}}{
   \setcounter{mathites}{0}
   \usecounter{mathites}
   \setlength{\topsep}{0pt plus 2pt minus 0pt}
   \setlength{\parskip}{0pt plus 2pt minus 0pt}
   \setlength{\partopsep}{0pt plus 2pt minus 0pt}
   \setlength{\parsep}{0pt plus 2pt minus 0pt}
   \setlength{\leftmargin}{22pt}
   \setlength{\itemsep}{0pt plus 2pt minus 0pt}}}
  {\end{list}}

\newcounter{mylist}
  {\begin{list}{{$\roman{mylist}$}}{
   \setlength{\topsep}{7pt plus 2pt minus 0pt}
   \setlength{\parsep}{3pt plus 2pt minus 0pt}
   \setlength{\leftmargin}{12pt}
   \setlength{\labelwidth}{-6pt}
   }
   }
  {\end{list}}

\newcounter{prostredi}
\def\theprostredi{\arabic{prostredi}}
\def\vejde#1{\unskip
\nobreak\hfill\penalty50\hskip1em\hbox{}\nobreak\hfill
\hbox{#1}}
\newenvironment{theorem}{\par\bigskip\noindent%
\refstepcounter{prostredi}{\bf Theorem \theprostredi.}\quad\bgroup\sl }
{\egroup\par\bigskip\endtrivlist}%
\newenvironment{lemma}{\par\bigskip\noindent%
\refstepcounter{prostredi}{\bf Lemma \theprostredi.}\quad\bgroup\sl }
{\egroup\par\bigskip\endtrivlist}%
\newenvironment{proof}{\par
\noindent%
{\bf Proof.}\quad\bgroup}
{\egroup\vejde{\rule{2.5mm}{2.5mm}}\par\bigskip\endtrivlist}%
\newenvironment{proofbt}{\par
\noindent%
{\bf Proof}\bgroup}
{\egroup\vejde{\rule{2.5mm}{2.5mm}}\par\bigskip\endtrivlist}%
\newenvironment{corollary}{\par\bigskip\noindent%
\refstepcounter{prostredi}{\bf Corollary
\theprostredi.}\quad\bgroup\sl }
{\egroup\par\bigskip\endtrivlist}%
\newenvironment{conjecture}{\par\bigskip\noindent%
\refstepcounter{prostredi}{\bf Conjecture
\theprostredi.}\quad\bgroup\sl }
{\egroup\par\bigskip\endtrivlist}%
\newcounter{prostrclaim}
\def\theprostrclaim{\arabic{prostrclaim}}
\def\vejde#1{\unskip
\nobreak\hfill\penalty50\hskip1em\hbox{}\nobreak\hfill
\hbox{#1}}
\newenvironment{claim}{\par\bigskip\noindent%
\refstepcounter{prostrclaim}{\underline{\bf Claim \theprostrclaim.}}\quad\bgroup\sl }
{\egroup\par\bigskip\endtrivlist}%
\newenvironment{proofcl}{\par
\noindent%
{\underline{Proof.}}\quad\bgroup}
{\egroup\vejde{$\Box$}\par\bigskip\endtrivlist}%
{\egroup\vejde{
$\Box$}\par\bigskip\endtrivlist}%

\newcounter{prostralph}
\def\theprostralph{\Alph{prostralph}}
\def\vejde#1{\unskip
\nobreak\hfill\penalty50\hskip1em\hbox{}\nobreak\hfill
\hbox{#1}}
\newenvironment{theoremAcite}[1]{\par\bigskip\noindent%
\refstepcounter{prostralph}{\bf Theorem
\theprostralph{} {#1}.}\quad\bgroup\sl }
{\egroup\par\bigskip\endtrivlist}%
\newenvironment{propositionAcite}[1]{\par\bigskip\noindent%
\refstepcounter{prostralph}{\bf Proposition
\theprostralph{} {#1}.}\quad\bgroup\sl }
{\egroup\par\bigskip\endtrivlist}%
\newenvironment{lemmaAcite}[1]{\par\bigskip\noindent%
\refstepcounter{prostralph}{\bf Lemma
\theprostralph{} {#1}.}\quad\bgroup\sl }
{\egroup\par\bigskip\endtrivlist}%
\newenvironment{conjectureAcite}[1]{\par\bigskip\noindent%
\refstepcounter{prostralph}{\bf Conjecture
\theprostralph{} {#1}.}\quad\bgroup\sl }
{\egroup\par\bigskip\endtrivlist}%

\newcommand{\sm}{\setminus}

\newcommand{\la}{\langle}
\newcommand{\lab}{\langle \{}
\newcommand{\ra}{\rangle}
\newcommand{\rag}{\rangle _G}

\newcommand{\rab}{\} \rangle}

\newcommand{\Gstx}{G^{^*}_x}

\newcommand{\ragstx}{\rangle_{G^{^*}_x}}

\newcommand{\bp}{\beginpicture}
\newcommand{\ep}{\endpicture}

\newcommand{\bs}{\bigskip}
\newcommand{\bsm}{\vspace{-4mm}}
\newcommand{\ms}{\medskip}

\begin{document}
\maketitle

\footnotetext[1]{All three authors are affiliated with the Department of Mathematics; 
European Centre of Excellence NTIS - New Technologies for the Information 
Society, University of West Bohemia,  Univerzitn\'{\i}~8, 301 00 Pilsen, 
Czech Republic.
E-mails: {\tt $\{$kabela,ryjacek,vranap$\}$@kma.zcu.cz.}}

\footnotetext[2]{The research was supported by the project LO1506 of the Czech 
Ministry of Education, Youth and Sports, and by the project P202/12/G061 
of the Czech Science Foundation.}


\begin{abstract}
We continue studying Thomassen's conjecture
(every $4$-connected line graph has a Hamilton cycle)
in the direction of a recently shown equivalence
with Jackson's conjecture
(every $2$-connected claw-free graph has a Tutte cycle),
and we extend the equivalent formulation as follows: 
In every connected claw-free graph, any two vertices are connected by a maximal path 
which is a Tutte path.

\ms

\end{abstract}


\section{Introduction}
\label{s-introduction}


In 1956, Tutte~\cite{Tutte} showed that every $4$-connected planar graph 
has a Hamilton cycle. The main ingredient of the proof is to consider a 
cycle $C$ in a graph $G$ such that every component of $G-V(C)$ has at most 
three neighbors on $C$ (then, if $G$ is $4$-connected, $C$ becomes 
a~Hamilton cycle).
Since then, Tutte cycles and paths have been often used in the study of 
Hamilton properties of graphs.
For instance, Thomassen~\cite{ThomassenSurf} showed that
all $4$-connected planar graphs are, in fact, Hamilton-connected.
It was also shown that the result of Tutte extends
to projective-planar graphs~\cite{ThoYu},
that every $4$-connected toroidal graph has a Hamilton path~\cite{ThoYuZan},
and that every $5$-connected such graph is Hamilton-connected~\cite{KawOze}.

We study Tutte paths in claw-free graphs in the direction of recent results 
on Tutte cycles and paths by \v{C}ada et al.~\cite{CadaEtAl} and 
Li et al.~\cite{LiEtAl}, which investigate the following conjecture 
formulated by Thomassen~\cite{Thomassen} (note that, for better orientation, 
quoted known results and statements are labeled with letters).

%
%
\begin{conjectureAcite}{\cite{Thomassen}}
\label{c-thomassen} 
Every $4$-connected line graph has a Hamilton cycle.
\end{conjectureAcite}

Conjecture~\ref{c-thomassen} remains open,
even though partial results and numerous equivalent formulations of this problem 
are known (see e.g. the survey~\cite{BroRyjVra}).
In the positive direction, Zhan~\cite{Zhan} and independently 
Jackson~\cite{Jackson89}
showed that every $7$-connected line graph of a multigraph has a Hamilton cycle, 
and Kaiser and the third author~\cite{KaiVra} showed that every
$5$-connected line graph of minimum degree at least $6$ is
Hamilton-connected.

In~\cite{Ryjacek}, the second author introduced a closure operation on 
claw-free graphs (a recursive application of the local completion operation 
on particular vertices that turns a claw-free graph into a line graph 
while preserving Hamiltonicity),
and showed that Conjecture~\ref{c-thomassen} is equivalent to a seemingly 
stronger conjecture by Matthews and Sumner~\cite{MatSum} saying that
every $4$-connected claw-free graph has a Hamilton cycle.

Two recent equivalent formulations of Conjecture~\ref{c-thomassen}
restate the problem in terms of Tutte cycles and Tutte paths.
Conjecture~\ref{c-jackson} was stated by Jackson~\cite{Jackson} and
the equivalence was shown by \v{C}ada et al.~\cite{CadaEtAl}
(using the closure technique of~\cite{Ryjacek}).
Recently, Li et al.~\cite{LiEtAl} showed that
the problem can be equivalently formulated as Conjecture~\ref{c-li}.

%
%
\begin{conjectureAcite}{\cite{Jackson}}
\label{c-jackson} 
Every $2$-connected claw-free graph has a Tutte cycle.
\end{conjectureAcite}

\bsm

%
%
\begin{conjectureAcite}{\cite{LiEtAl}}
\label{c-li}
For every pair of vertices $a, b$ of a connected line graph of a hypergraph of 
rank at most $3$, there is a maximal $(a, b)$-path which is a Tutte path.
\end{conjectureAcite}

As the main result, we present the following reformulation of the problem
(seemingly stronger than Conjecture~\ref{c-jackson}),
and we show that they are, in fact, equivalent.

%
%
\begin{conjecture}
\label{c-main}  
For every pair of vertices $a, b$ of a connected claw-free graph,
there is a maximal $(a, b)$-path which is a Tutte path.
\end{conjecture}

\bsm

%
%
\begin{theorem}
\label{t-eqivalence}
Conjectures~\ref{c-thomassen} and~\ref{c-main} are equivalent.
\end{theorem}

For the proof of Theorem~\ref{t-eqivalence},
we introduce a Tutte-closure operation (for details, see 
Section~\ref{s-preliminaries}), 
preserving the property that every pair 
of vertices of a graph is joined by a maximal path which is a Tutte path
(a similar concept can be found in~\cite{KuzRyjTes});
and we show the following.

%
%
\begin{theorem}
\label{t-closure}
A Tutte-closure of a claw-free graph
is the line graph of a hypergraph of rank at most $3$.
\end{theorem}

In relation to Conjectures~\ref{c-li} and \ref{c-main}, we  note that there 
are infinitely many claw-free graphs which are not line 
graphs of hypergraphs of rank at most $3$, and vice versa (this follows e.g. 
from \cite[Theorem~1]{JavMalOmo}, and the vice versa is trivial).

\ms

In Section~\ref{s-preliminaries}, we show that Theorem~\ref{t-eqivalence}
follows from Theorem~\ref{t-closure} (at the beginning of 
Section~\ref{s-preliminaries}, we recall the concepts used). 
As the main ingredients for proving Theorem~\ref{t-closure}, we use two 
characterizations of line graphs of multigraphs (discussed
in Section~\ref{s-line graphs}) in combination with properties of 2-closed 
claw-free graphs (Section~\ref{s-2-closed}) and properties of maximal 
$(a,b$)-paths (Section~\ref{s-paths}). 
The proof of Theorem~\ref{t-closure} is given in Section~\ref{s-proof}.


\section{Preliminaries}
\label{s-preliminaries}


We first recall definitions of the concepts used (some notation will be recalled
later in the text as needed). For additional concepts and notation, we refer the
reader to \cite{BonMur}.

A graph is said to be \emph{claw-free} if it contains no copy of the graph 
$K_{1,3}$ as an induced subgraph.
A \emph{Hamilton path} (\emph{Hamilton cycle}) in a graph is a path (cycle) 
containing all its vertices.
Considering a path $P$ in a graph $G$,
we say that $P$ is a \emph{Tutte path}
if every component of $G-V(P)$ has at most $\min\{3,|V(P)|-1\}$ neighbors on $P$.
In particular, a Hamilton path is a Tutte path.
Similarly, a cycle $C$ is a \emph{Tutte cycle} 
if $C$ is a Hamilton cycle or
$|V(C)| \geq 4$ and every component of $G-V(C)$ has at most three neighbors on~$C$.

We recall that an \emph{$(a,b)$-path} is a path from vertex $a$ to vertex $b$, 
and we say that $P$ is a \emph{maximal $(a,b)$-path}
if there is no $(a,b)$-path $P'$ such that $V(P)$ is a proper subset of $V(P')$.
We say that a graph is \emph{Tutte-connected} (\emph{Hamilton-connected})
if for every pair of its vertices $a, b$,
there is a maximal $(a, b)$-path which is a Tutte path
(an $(a, b)$-path which is a Hamilton path), respectively
(thus, if $G$ is not Tutte-connected,
then there is a pair of vertices $a, b$ such that
no maximal $(a,b)$-path in $G^T$ is a Tutte path).
%

For $x\in V(G)$, $N_G(x)$ denotes the {\em neihgborhood of $x$ in $G$}, i.e. 
the set of all vertices of $G$ that are adjacent to $x$,
$N_G[x]$ denotes the set $N_G(x) \cup \{x\}$, and for a set $A\subset V(G)$,
$\langle A \rangle_{G}$ denotes the subgraph of $G$ induced by $A$.
The \emph{local completion} of a graph $G$ at a vertex $x\in V(G)$ is the 
graph $\Gstx$ obtained from $G$ by adding all possible edges among vertices 
of $N_G(x)$.
In~particular, the graph $\langle N_{\Gstx}(x)\rangle_{\Gstx}$ is a 
\emph{clique}, meaning a complete subgraph (not necessarily maximal).  
The edges in $E(\la N_{\Gstx}(x)\ragstx)\sm E(\la N_G(x)\rag)$ will be 
sometimes referred to as {\em new edges}.

We define a {\em Tutte-closure} of a graph $G$ to be a graph $G^T$ obtained 
as follows.
If $G$ is Tutte-connected, then we let $G^T$
be the complete graph obtained by adding all possible edges.
Otherwise, we recursively perform the
local completion operation at vertices for which the resulting graph
remains not Tutte-connected, as long as there is at least one such vertex;
and we let $G^T$ be a resulting graph.
We say that $G^T$ is a Tutte-closure of $G$.
(More precisely, there is a sequence of graphs $G_0,G_1,\ldots,G_k$
such that $G_0=G$, $G_{i+1}=(G_i)^{^*}_{x_i}$ for some $x_i\in V(G_i)$ such 
that $(G_i)^{^*}_{x_i}$ is still not Tutte-connected, and $(G_k)^{^*}_{x_k}$ 
is Tutte-connected for any nonsimplicial vertex $x_k\in V(G_k)$, and we set 
$G_k=G^T$). 

Note that it can happen that, for some two vertices $x,y$ of a 
non-Tutte-connected graph $G$, neither $(G)^{^*}_{x}$ nor $(G)^{^*}_{y}$
is Tutte-connected, while the graph $((G)^{^*}_{x})^{^*}_y $ is. 
Such situations indicate that a Tutte-closure of a non-Tutte-connected graph is 
not uniquely determined.

Although the main result itself deals with graphs with neither loops nor 
multiple edges, we will need line graphs of multigraphs and hypergraphs for 
its proof.
In accordance with~\cite{LiEtAl},
we define a \emph{hypergraph} as a collection of subsets (edges) of a ground set.
The \emph{rank} of a hypergraph is the size of its maximum edge.
(In~particular, multiple edges are allowed
and \emph{multigraphs} are exactly hypergraphs of rank at most $2$.)
The \emph{line graph} of a graph (of a multigraph, of a hypergraph) $H$
is the graph whose vertex set is $E(H)$, and two of its vertices are adjacent
if and only if the corresponding edges in $H$ have a vertex in common
(if we speak only about a line graph, we mean the line graph of a graph.)

\ms

We conclude this section by showing that Theorem~\ref{t-closure} implies 
Theorem~\ref{t-eqivalence}.

\bs

\begin{proofbt} {\bf of Theorem~\ref{t-eqivalence}.} \quad
First, suppose that Conjecture~\ref{c-li} is true.
We consider a connected claw-free graph $G$ and its Tutte-closure $G^T$.
Clearly, $G^T$ is connected.
By Theorem~\ref{t-closure}, $G^T$ is a line graph of a hypergraph of rank at 
most $3$.
By the assumption that Conjecture~\ref{c-li} is true, $G^T$ is 
Tutte-connected, and so is $G$ (since the closure preserves the 
property), and thus Conjecture~\ref{c-main} is true.

Conversely, suppose that Conjecture~\ref{c-main} is true.
We consider a $4$-connected line graph $G$
and an adjacent pair of its vertices, say, $a, b$.
Clearly, $G$ is claw-free.
We note that the $(a,b)$-path given by Conjecture~\ref{c-main}
is, in fact, a Hamilton path.
We extend this path by adding the edge $ab$,
and we conclude that Conjecture~\ref{c-thomassen} is true.
\end{proofbt}


\section{Line graphs of multigraphs (hypergraphs)}
\label{s-line graphs}


An important ingredient of the argument is using two different
characterizations of line graphs of multigraphs.
The first characterization, due to Bermond and Meyer~\cite{BerMey},
describes the class in terms of forbidden induced subgraphs.
%

%
%
\begin{theoremAcite}{\cite{BerMey}}
\label{t-multigraphs} 
For $i=1, 2, \dots, 7$, let $G_i$ be the graph in Figure~\ref{f-subgraphs}.
A~graph is a line graph of a multigraph if and only if it is
$\{G_1, G_2, \dots, G_7\}$-free.
\end{theoremAcite}

%
%
\begin{figure}[ht]
\bsm
$$\bp
\setcoordinatesystem units <0.75mm,0.75mm>
\setplotarea x from -50 to 50, y from 0 to 5
\put{
\bp
\setcoordinatesystem units <0.6mm,.5mm>
\setplotarea x from -20 to 20, y from -20 to 15
\put{$\bullet$} at -10  15
\put{$\bullet$} at  10  15
\put{$\bullet$} at   0   0
\put{$\bullet$} at   0 -15
\plot -10 15  0 0  10 15 /
\plot 0 0  0 -15 /
\put{$G_1$} at  0 -24
\ep} at -95 -1
\put{
\bp
\setcoordinatesystem units <0.6mm,.5mm>
\setplotarea x from -20 to 20, y from -20 to 15
\put{$\bullet$} at -10  15
\put{$\bullet$} at  10  15
\put{$\bullet$} at -10   0
\put{$\bullet$} at  10   0
\put{$\bullet$} at -10 -15
\put{$\bullet$} at  10 -15
\plot -10 -15 -10 15 10 15 10 -15 -10 -15 10 0 -10 0 10 15 /
\put{$G_2$} at  0 -24 
\ep} at -72 -1
\put{
\bp
\setcoordinatesystem units <0.5mm,.5mm>
\setplotarea x from -20 to 20, y from -20 to 15
\put{$\bullet$} at   0  15
\put{$\bullet$} at   0   0
\put{$\bullet$} at -20   0
\put{$\bullet$} at  20   0
\put{$\bullet$} at -10 -15
\put{$\bullet$} at  10 -15
\plot 0 15 -20 0 -10 -15 10 -15 20 0 0 15 0 0 -10 -15 /
\plot -20 0 20 0 /
\plot 0 0  10 -15 /
\put{$G_3$} at 0 -24
\ep} at -45 -1
\put{
\bp
\setcoordinatesystem units <0.4mm,.55mm>
\setplotarea x from -20 to 20, y from -20 to 20
\put{$\bullet$} at -25   0
\put{$\bullet$} at  -5   0
\put{$\bullet$} at   5   0
\put{$\bullet$} at  25   0
\put{$\bullet$} at   5  15
\put{$\bullet$} at  -5 -15
\plot -25 0 25 0  5 15 -25 0  -5 -15 5 0  5 15  -5 0 /
\plot -5 -15  25 0 /
\put{$G_4$} at 0 -22
\ep} at  -12 0
\put{
\bp
\setcoordinatesystem units <0.45mm,.6mm>
\setplotarea x from -20 to 20, y from -20 to 20
\put{$\bullet$} at   0  15
\put{$\bullet$} at   0 -15
\put{$\bullet$} at -25   0
\put{$\bullet$} at -10   0
\put{$\bullet$} at  10   0
\put{$\bullet$} at  25   0
\plot -25 0  25 0  0 15 -25 0  0 -15 -10 0  0 15  10 0  0 -15 25 0 /
\put{$G_5$} at 0 -21
\ep} at 23 0
\put{
\bp
\setcoordinatesystem units <0.5mm,.4mm>
\setplotarea x from -20 to 20, y from -20 to 20
\put{$\bullet$} at   5  20
\put{$\bullet$} at   0   0
\put{$\bullet$} at  10   0
\put{$\bullet$} at   5  -5
\put{$\bullet$} at   5 -12
\put{$\bullet$} at -20 -25
\put{$\bullet$} at  30 -25
\plot 5 20 -20 -25  30 -25  5 20  0 0  -20 -25  5 -12  30 -25  10 0
    5 20  5 -12  0 0  10 0  5 -12 /
\plot -20 -25  5 -5  0 0 /
\plot 30 -25  5 -5  10 0 /
\put{$G_6$} at 0 -36
\ep} at   55 -1
\put{
\bp
\setcoordinatesystem units <0.6mm,.5mm>
\setplotarea x from -20 to 20, y from -20 to 20
\put{$\bullet$} at   0  20
\put{$\bullet$} at   0   5
\put{$\bullet$} at -20  -5
\put{$\bullet$} at  20  -5
\put{$\bullet$} at  -5  -5
\put{$\bullet$} at   5  -5
\put{$\bullet$} at  0  -15
\plot 0 20 -20 -5  20 -5  0 20  0 -15 -20 -5  0 5  20 -5
   0 -15  -5 -5  0 5  5 -5  0 -15 /
\setquadratic
\plot -20 -5  0 -21  20 -5 /
\setlinear
\put{$G_7$} at 0  -27
\ep} at   90 0
\ep$$
\bsm\bsm
\caption{Forbidden induced subgraphs for line graphs of multigraphs}
\label{f-subgraphs}
\end{figure}

The second characterization, due to Krausz~\cite{Krausz}, describes line 
graphs of hypergraphs and of multigraphs in terms of clique coverings.
Viewing multigraphs as hypergraphs of rank at most $2$,
we will use the following generalization of Krausz's characterization 
(see~\cite{Berge, SkuSuzTys}).

%
%
\begin{theoremAcite}{\cite{Berge,SkuSuzTys}}
\label{t-hypergraphs} 
For every positive integer $r$,
a graph $G$ is a line graph of a hypergraph of rank at most $r$ if and only if
$E(G)$ can be covered by a system $\mathcal{K}$ of cliques
such that every vertex of $G$ belongs to at most $r$ cliques of $\mathcal{K}$.
\end{theoremAcite}

The basic idea of the proof of Theorem~\ref{t-closure}
is to find a suitable clique covering of a Tutte-closure
and to use the case $r = 3$ of Theorem~\ref{t-hypergraphs} (for hypergraphs 
of rank at most $3$).
In finding such covering, we will also use the case $r = 2$ (for multigraphs).

Considering line graphs of hypergraphs of rank at most~$r$,
we remark that the terminology is not unified; these graphs are also referred 
to as graphs of $\infty$-Krausz dimension at most $r$~\cite{GleMatSku},
 as $r$-set representations~\cite{PolRodTur},
as graphs with clique cover number~$r$~\cite{JavMalOmo},
or as edge intersection graphs of $r$-uniform hypergraphs~\cite{SkuSuzTys}.
For an overview of the terminology used and for more details,
we refer the reader to~\cite{JavMalOmo, SkuSuzTys}.


\section{On 2-closed claw-free graphs}
\label{s-2-closed}


In this section, we note that a Tutte-closure of a claw-free graph
is, in fact, claw-free.
In addition, we suppose that it is $2$-closed
(the definition is to be found below),
and we construct a desired clique covering.
To this end, we will use several lemmas on $2$-closed claw-free graphs
(proved by the second and third authors in~\cite{RyjVraLG}).

We note the following
(similar statements are shown, for instance, in~\cite{Ryjacek, BroRyj}).

%
%
\begin{lemma}
\label{l-free}
Let $G$ be a $K_{1,k}$-free graph (where $k \geq 3$).
For every vertex $x\in V(G)$, the graph $\Gstx$ is $K_{1,k}$-free.  
\end{lemma}

\begin{proof}
We consider a vertex $z$ and a maximum independent set $I$ of 
$\langle N_{\Gstx}(z) \rangle_{\Gstx}$, and we show that 
$\langle N_G(z) \rangle_G$ contains an independent set of size $|I|$
(and thus $|I|< k$).
Clearly, $I$ is an independent set in $G$.
If $I\subset N_G(z)$, we are done, hence we can assume that there is a vertex, 
say $v$, of $I \sm N_G(z)$. We observe that $x$ is adjacent to $v$ and $z$ and 
to no vertex of $I \sm \{ v \}$, and thus $(I \sm \{ v \}) \cup \{ x \}$ is an 
independent set in $\langle N_G(z) \rangle_G$. 
\end{proof}

Following~\cite{BollobasEtAl}, we recall the concept of a $k$-closure.
For a positive integer $k$ and a graph $G$, the \emph{$k$-closure} of $G$
is the graph obtained from $G$ 
by recursively performing the local completion operation at
vertices whose neighborhood induces a non-complete $k$-connected
graph, as long as this is possible.
We say that $G$ is \emph{$k$-closed} if $G$ is isomorphic to its $k$-closure.
Note that it was shown in~\cite{BollobasEtAl} that the $k$-closure of a graph
is uniquely determined.

\ms

Comparing Theorem~\ref{t-multigraphs} to the following lemma,
we note a similarity in the structure of 
line graphs of multigraphs and $2$-closed claw-free graphs.

%
%
\begin{lemmaAcite}{\cite{RyjVraLG}}
\label{l-forbidden}
For $i=1, 3, 5, 6, 7$, let $G_i$ be the graph in Figure~\ref{f-subgraphs}.
Every $2$-closed claw-free graph is $\{G_1,G_3,G_5,G_6,G_7\}$-free.
\end{lemmaAcite}

In particular, every $\{G_2,G_4\}$-free induced subgraph of a $2$-closed 
claw-free graph is a line graph of a multigraph (hence can be covered by a 
system of cliques given by the case $r = 2$ of Theorem~\ref{t-hypergraphs}).
We view this subgraph as a starting point for the construction of a desired clique covering.

\ms

To manage the graphs $G_2$ and $G_4$, we use a concept of good walks (introduced in~\cite{RyjVraLG}).
Recall that a \emph{walk} is a sequence of vertices such that every two 
consecutive vertices are adjacent.
Considering a walk $J=u_0 u_1 \dots u_{k+1}$ in a graph $G$,
we say that $J$ is \emph{good} if the following conditions are satisfied:
\begin{itemize}
\item $k\geq 4$,
\item if $0\leq i \leq k-1$, then the vertices $u_i$ and $u_{i+2}$ are adjacent 
     in $G$,
\item if $0\leq i \leq k-4$, then the graph 
     $\la\{u_i,u_{i+1},\ldots,u_{i+5}\}\rag$ is a copy of either $G_2$ or $G_4$ 
     (see Fig.~\ref{f-subgraphs}). 
\end{itemize}
We say that a good walk $J$ is \emph{maximal} if there is no good walk
$J'$ in $G$ such that $J$ is a proper subwalk of $J'$
(that is, $J$ being a subsequence of $J'$ implies that $J$ is $J'$).

We recall that the \emph{square} of a graph $H$ is the graph on the same vertex 
set in which two vertices are adjacent if and only if their distance in $H$ is 
at most~$2$.
We will use the following properties of good walks (proved in~\cite{RyjVraLG}).

%
%
\begin{lemmaAcite}{\cite{RyjVraLG}}
\label{l-path}
Let $G$ be a connected $2$-closed claw-free graph that is not the square of a cycle.
If $J=u_0 u_1 \dots u_{k+1}$ is a good walk in $G$, then $u_1 u_2 \dots u_k$ is a path.
\end{lemmaAcite}

%
%
\begin{lemmaAcite}{\cite{RyjVraLG}}
\label{l-disjoint}
Let $G$ be a connected $2$-closed claw-free graph that is not the square of a cycle.
If $J=u_0 u_1 \dots u_{k+1}$ and $J=u'_0 u'_1 \dots u'_{k'+1}$
are maximal good walks in $G$ such that
$u_s=u'_t$, for some $1 \leq s \leq k$ and $1 \leq t \leq k'$,
then the following statements are satisfied:
\begin{enumerate}
\item[$(1)$] $\{u_1, u_2, \dots, u_k\} = \{u'_1, u'_2, \dots, u'_{k'} \}$,
\item[$(2)$] $k = k'$, and $u_i = u'_i$ or $u_i = u'_{k-i+1}$ for every $i = 1, 2, \dots, k$.
\end{enumerate}
\end{lemmaAcite}

%
%
\begin{lemmaAcite}{\cite{RyjVraLG}}
\label{l-clique}
Let $G$ be a connected 2-closed claw-free graph that is not the square of a cycle.
If $J=u_0 u_1 \dots u_{k+1}$ is a maximal good walk in $G$, then
$N_G[u_1] \setminus \{u_3\} = N_G[u_2] \setminus \{u_3,u_4\}$, 
and this set induces a clique in $G$.
\end{lemmaAcite}

%
%
\begin{lemmaAcite}{\cite{RyjVraLG}}
\label{l-degree}
Let $G$ be a connected 2-closed claw-free graph.
If $J=u_0 u_1 \dots u_{k+1}$ is a good walk in $G$ such that $k \geq 5$,
then for every $i=3,4,\dots,k-2$, $u_i$ has degree $4$ in $G$.
\end{lemmaAcite}

In addition, we recall the following lemma, also proved in~\cite{RyjVraLG}
(it will be used, besides this section, in Section~\ref{s-proof}).

%
%
\begin{lemmaAcite}{\cite{RyjVraLG}}
\label{l-non-adjacent}
Let $x$ be a vertex of a claw-free graph $G$. If there is a $2$-connected 
subgraph of $\langle N_G(x)\rangle_G$ containing two disjoint pairs of 
non-adjacent vertices, then $\langle N_G(x)\rangle_G$ is $2$-connected.
\end{lemmaAcite}

We combine the results discussed in 
Sections~\ref{s-line graphs}~and~\ref{s-2-closed},
and we obtain a system of cliques
covering a $2$-closed claw-free graph as follows.

%
%
\begin{lemma}
\label{l-cover}
If $G$ is a $2$-closed claw-free graph,
then $E(G)$ can be covered by a system $\mathcal{K}$ of cliques 
such that every vertex of $G$ belongs to at most three cliques of~$\mathcal{K}$.
Furthermore, if a vertex belongs to three cliques of $\mathcal{K}$, then
these are the only maximal cliques of $G$ containing this vertex.
\end{lemma}

\begin{proof}
We can assume that $G$ is connected
(otherwise, we apply the argument to each component of $G$).
In addition, we observe that if $G$ is the square of a cycle, then it has the desired covering;
and so we can assume that it is not.

For $i=1, 2, \dots, 7$, let $G_i$ be the graph in Figure~\ref{f-subgraphs},
and recall that $G$ is $\{G_1,G_3,G_5,G_6,G_7\}$-free by Lemma~\ref{l-forbidden}. 
By definition, every induced copy of $G_2$ or $G_4$ gives a good walk 
(see Figure~\ref{f-walk});
and thus for each such induced subgraph, all its vertices are contained in a 
maximal good walk.


%
%
%
\begin{figure}[ht]
$$\bp
\setcoordinatesystem units <0.8mm,1mm>
\setplotarea x from -30 to 30, y from -7 to 7
\put{
\bp
\setcoordinatesystem units <1.0mm,0.8mm>
\setplotarea x from -20 to 20, y from -10 to 10
\put{$\bullet$} at  -20  -8 
\put{$\bullet$} at  -12   8 
\put{$\bullet$} at   -4  -8 
\put{$\bullet$} at    4   8 
\put{$\bullet$} at   12  -8 
\put{$\bullet$} at   20   8 
\plot -20 -8  -12 8  -4 -8  4 8  12 -8  20 8 /
\plot -20 -8  12 -8 /
\plot -12  8  20  8 /
\put{$u_0$} at  -20  -11
\put{$u_1$} at  -12   11
\put{$u_2$} at   -4  -11
\put{$u_3$} at    4   11
\put{$u_4$} at   12  -11
\put{$u_5$} at   20   11
\put{$G_2$} at  -20   8
\ep} at -35 0
\put{
\bp
\setcoordinatesystem units <1.0mm,0.8mm>
\setplotarea x from -20 to 20, y from -10 to 10
\put{$\bullet$} at  -20  -8 
\put{$\bullet$} at  -12   8 
\put{$\bullet$} at   -4  -8 
\put{$\bullet$} at    4   8 
\put{$\bullet$} at   12  -8 
\put{$\bullet$} at   20   8 
\plot -20 -8  -12 8  -4 -8  4 8  12 -8  20 8 /
\plot -20 -8  12 -8 /
\plot -12  8  20  8 /
\put{$u_0$} at  -20  -11
\put{$u_1$} at  -12   11
\put{$u_2$} at   -4  -11
\put{$u_3$} at    4   11
\put{$u_4$} at   12  -11
\put{$u_5$} at   20   11
\setquadratic
\plot -20 -8  -12  13.5     20 8 /
\setlinear
\put{$G_4$} at  -22   8
\ep} at  35 0
\ep$$
\bsm
\caption{The graphs $G_2$ and $G_4$ viewed as good walks.}
\label{f-walk}
\end{figure}

By Lemma~\ref{l-path}, we have that
if $u_0 u_1 \dots u_{k+1}$ is a maximal good walk in $G$,
then $u_1 u_2 \dots u_k$ is a path;
and we consider a set $\mathcal{P}$ of all these paths
(meaning that the paths of $\mathcal{P}$ are pairwise
distinct; in particular, if $u_1 u_2 \dots u_k$ is in $\mathcal{P}$, then 
$u_k u_{k-1} \dots u_1$ is not).
By Lemma~\ref{l-disjoint}, the paths in $\mathcal{P}$ are pairwise vertex-disjoint.

For every path of $\mathcal{P}$, we remove from $G$ all its interior vertices 
(that is, vertices $u_2,\ldots,u_{k-1}$);
and we let $G'$ denote the resulting graph.
We note that $G'$ is $\{G_1, G_2, \dots, G_7\}$-free,
and thus it is a line graph of a multigraph by Theorem~\ref{t-multigraphs};
and we consider a system, say, $\mathcal{K}_{1}$, of cliques
given by Theorem~\ref{t-hypergraphs}
(covering $E(G')$ such that every vertex belongs to at most two cliques). 

In addition, we consider the system, say, $\mathcal{K}_{2}$,
of all distinct cliques in $G$ given by Lemma~\ref{l-clique}
(that is, $\mathcal{K}_{2}$ contains $\la N_G[u_1] \sm \{u_3\}\rag$
for every endvertex $u_1$ of a path in $\mathcal{P}$,
and $\la N_G[u_k] \sm \{u_{k-2}\}\rag$ for every $u_k$,
see Fig.~\ref{f-goodwalk}.)
For every path in $\mathcal{P}$,
we remove from $\mathcal{K}_{1}$ all cliques containing its endvertex
(that is, $u_1$ or $u_k$),
and then we add all cliques of $\mathcal{K}_{2}$;
and we let $\mathcal{K}_{0}$ denote the resulting system.

%
%
%
\begin{figure}[ht]
$$\bp
\setcoordinatesystem units <0.8mm,1mm>
\setplotarea x from -30 to 30, y from -7 to 7
\put{
\bp
\setcoordinatesystem units <1.0mm,0.8mm>
\setplotarea x from -20 to 20, y from -14 to 14
\put{$\bullet$} at  -32  -8 
\put{$\bullet$} at  -24   8 
\put{$\bullet$} at  -16  -8 
\put{$\bullet$} at   -8   8 
\put{$u_0$} at  -32  -11
\put{$u_1$} at  -24   11
\put{$u_2$} at  -16  -11
\put{$u_3$} at   -8   11
\plot -32 -8  -24 8  -16 -8  -8 8  -6 4 /
\plot -24  8   -4 8 /
\plot -32 -8  -10 -8 /
\circulararc  360  degrees from  -12  0 center at  -26 -2
\put{$K$} at  -32   2
\put{$\bullet$} at    8  -8 
\put{$\bullet$} at   16   8 
\put{$\bullet$} at   24  -8 
\put{$\bullet$} at   32   8 
\put{$u_{k-2}$} at    8  -11
\put{$u_{k-2}$} at   16   11
\put{$u_k$}     at   24  -11
\put{$u_{k+1}$} at   32   11
\plot  32  8   24 -8   16 8   8 -8   6 -4 /
\plot  24 -8    4 -8 /
\plot  32  8   10  8 /
\put{$\dots$} at  4  8 
\put{$\dots$} at -4 -8 
\circulararc  360  degrees from   12  0 center at  26 2
\put{$K'$} at   32  -2
\ep} at  0 0
\ep$$
\caption{A good walk $u_0u_1\ldots u_{k+1}$ and two cliques $K,K'\in\mathcal{K}_2$.}
\label{f-goodwalk}
\end{figure}

We call an edge \emph{red} if it is the first (last) edge of a path of 
$\mathcal{P}$.
For every vertex $v$ of $G$, the properties of good walks 
(Lemmas~\ref{l-disjoint},~\ref{l-clique},~\ref{l-degree}) and the definitions 
of the systems of cliques $\mathcal{K}_1$, $\mathcal{K}_2$ and $\mathcal{K}_0$
immediately imply the following:
\begin{itemize}
\item  if $v$ is in no path of $\mathcal{P}$, then it belongs to at most 
     two cliques of~$\mathcal{K}_{0}$,
\item if $v$ is incident with a red edge, then it belongs to precisely 
     one clique of~$\mathcal{K}_{0}$,
\item otherwise, it belongs to none.
\end{itemize}
Furthermore, Lemma~\ref{l-clique} implies that for any path $u_1\ldots u_k$ 
in $\mathcal{P}$, $N_G[u_1]\sm\{u_3\}=N_G[u_2]\sm\{u_3,u_4\}$, and
$N_G[u_k]\sm\{u_{k-2}\}=N_G[u_{k-1}]\sm\{u_{k-2},u_{k-3}\}$.
By Lemma~\ref{l-degree},
if $3 \leq i \leq k-2$, then $u_i$ has degree $4$ in $G$,
and thus $N_G(u_i) = \{ u_{i-2}, u_{i-1}, u_{i+1}, u_{i+2} \}$
by the definition of a good walk.
Consequently, the only edges of $G$ 
which are not covered by $\mathcal{K}_{0}$ are the non-red edges
whose both ends belong to a common path of $\mathcal{P}$
(namely, $u_i u_{i+1}$ for $i=2, \dots, k-2$ and
$u_i u_{i+2}$ for $i=1, 2, \dots, k-2$). 

By definition, every three consecutive vertices of a path of $\mathcal{P}$
induce a clique (namely, a triangle);
and we extend $\mathcal{K}_{0}$ by adding all these cliques,
and we let $\mathcal{K}$ denote the resulting system.
Clearly, every edge of $G$ is covered by a clique of~$\mathcal{K}$. 
We note that every vertex of $G$ belongs to at most three cliques of $\mathcal{K}$,
and furthermore a vertex belongs to three cliques
if and only if it is an interior vertex of a path of $\mathcal{P}$.
It remains to show that for each such vertex, these cliques are maximal and, in fact,
the only maximal cliques containing this vertex.

We consider an interior vertex $u_i$ of a path of $\mathcal{P}$,
and we discuss two cases.
In case $3 \leq i \leq k-2$, the considered cliques are
$\langle \{ u_{i-2}, u_{i-1}, u_{i} \} \rangle_{G}$,
$\langle \{ u_{i-1}, u_{i}, u_{i+1} \} \rangle_{G}$ and
$\langle \{ u_{i}, u_{i+1}, u_{i+2} \} \rangle_{G}$.
We recall that $N_G(u_i) = \{ u_{i-2}, u_{i-1}, u_{i+1}, u_{i+2} \}$
and $\langle N_{G}[u_i] \rangle_{G}$ is an induced subgraph of $G_2$ or $G_4$.
The desired claim follows.

Otherwise, we suppose that $i = 2$ (the argument for $i = k-1$ is similar).
We recall that the considered cliques are 
$\langle N_G[u_2] \setminus \{u_3,u_4\}\rangle_G$,
$\langle \{ u_{1}, u_{2}, u_{3} \} \rangle_{G}$ and
$\langle \{ u_{2}, u_{3}, u_{4} \} \rangle_{G}$.
Thus, we need to show that $u_,u_3$ and $u_4$ have no other common neighbors, 
i.e., that 
$N_G(u_2) \cap N_G(u_3) = \{u_1,u_4\}$ and
$N_G(u_2) \cap N_G(u_4) = \{u_3\}$
(recall that, by definition, $u_3$ is not adjacent to $u_0$,
and $u_4$ is adjacent to neither $u_0$ nor $u_1$).
For the sake of a contradiction, suppose that there is a vertex, say $v$, of
$N_G(u_2) \sm \{u_0,u_1,u_3,u_4\}$ which is adjacent to $u_3$ or $u_4$.
Lemma~\ref{l-clique} implies that $v$ is adjacent to $u_0$ and $u_1$,
and also to both $u_3$ and $u_4$ (by considering the reversed walk).
We note that 
$\langle \{ u_{0}, u_{1}, u_{3}, u_{4}, v \} \rangle_{G}$
is a $2$-connected subgraph of $\langle N_G(u_{2}) \rangle_{G}$
and $u_{0}, u_{3}$ and $u_{1}, u_{4}$ are two pairs of non-adjacent vertices.
By Lemma~\ref{l-non-adjacent}, $\langle N_{G}(x)\rangle_{G}$ is $2$-connected,
contradicting the assumption that $G$ is $2$-closed.
\end{proof}


\section{Local completions and maximal (a,b)-paths}
\label{s-paths}


With Lemma~\ref{l-cover} on hand, we can focus on vertices
whose neighborhood is non-complete and $2$-connected.
We recall results of~\cite{BollobasEtAl,RyjVraStability}
dealing with long paths in a local completion of a claw-free graph,
and we combine and adapt them for maximal $(a,b)$-paths.

%
%
\begin{propositionAcite}{\cite{BollobasEtAl}}
\label{p-long} 
Let $a, b$ and $x$ be vertices of a claw-free graph $G$ such that $N_G(x)$ 
induces a $3$-connected graph in $G$.
If $\Gstx$ has an $(a,b)$-path of length $\ell$, then $G$ has an $(a,b$)-path of 
length at least $\ell$. 
\end{propositionAcite}

Following~\cite{RyjVraStability}, for $i=0,1,2$ and a given vertex $x$ and a
path $P$ in a graph $G$, let $V^{x}_{i}(P)$ denote the set of all vertices
$y$ of $V(P) \cap N_G(x)$ such that $|N_{P}(y) \cap N_G[x]|=i$.
If $V^{x}_1(P)$ is non-empty, then $a^x_P$ $(b^x_P)$ denotes the first (last) 
vertex on $P$ belonging to $V^x_1(P)$, respectively.

We recall that a \emph{cut} in a connected graph is a set of vertices whose 
removal results in a disconnected graph.

%
%
\begin{propositionAcite}{\cite{RyjVraStability}}
\label{p-longest} 
Let $a, b$ and $x$ be vertices of a claw-free graph $G$ such that the connectivity
of $\langle N_G(x)\rangle_G$ is $2$, and let $R$ be a minimum cut in $\langle N_G(x)\rangle_G$.
Let $P^*$ be a longest $(a,b)$-path in $\Gstx$.
Then every longest $(a,b)$-path in $G$ is shorter than $P^*$ if and only if
$\{a,b\}$ is a cut in $\langle N_G(x) \rangle_G$
and every longest $(a,b)$-path $P$ in $\Gstx$ has the following properties:
\begin{enumerate}
\item[$(1)$] $x$ belongs to $V(P)$,
\item[$(2)$] $a,a^x_P,b,b^x_P$ are distinct vertices,
\item[$(3)$] $a^x_Pb^x_P$ is an edge of $G$,
\item[$(4)$] every component of ${\langle N_G(x) \rangle}_{G} -R$
     contains a vertex not belonging to $V^x_0(P)$,
\item[$(5)$] for every pair of vertices $u,v$ of $V^x_1(P)$ such that $u,v$ are in different components 
     of ${\langle N_G(x) \rangle}_{G} -R$, if the $(u,v)$-subpath of $P$ 
     contains an interior vertex, then it contains a vertex of $N_G[x] \setminus V^x_0(P)$
     as an interior vertex.
\end{enumerate}
Moreover, if a longest $(a,b)$-path in $G$ is shorter than $P^*$, then
the following are satisfied:
\begin{enumerate} \setcounter{enumi}{5}
\item[$(6)$] for every longest $(a,b)$-paths $P$ and $Q$ in $\Gstx$, we have
     $a^x_P=a^x_Q$ and $b^x_P=b^x_Q$,
\item[$(7)$] $a^x_Pa$ and $b^x_Pb$ are edges of $G$.
\end{enumerate}
\end{propositionAcite}

As a corollary of Propositions~\ref{p-long} and~\ref{p-longest},
we prove the following statement 
(which is formally stronger since longest paths are maximal).

%
%
\begin{corollary}
\label{c-maximal} 
Let $a, b$ and $x$ be vertices of a claw-free graph $G$ such that $N_G(x)$ induces a
$2$-connected graph in $G$, and let $R$ be a minimum cut in ${\langle N_G(x) \rangle}_{G}$.
Let $P^*$ be a maximal $(a,b)$-path in $\Gstx$.
Then there is no $(a,b)$-path $P'$ in $G$ such that $V(P^*)=V(P')$ if and only if
$\{a,b\}$ is a cut in $\langle N_G(x) \rangle_G$
and every $(a,b)$-path $P$ in $\Gstx$ such that $V(P^*)=V(P)$ has the following properties:
\begin{enumerate}
\item[$(1)$] $x$ belongs to $V(P)$,
\item[$(2)$] $a,a^x_P,b,b^x_P$ are distinct vertices,
\item[$(3)$] $a^x_Pb^x_P$ is an edge of $G$,
\item[$(4)$] every component of ${\langle N_G(x) \rangle}_{G} -R$
     contains a vertex not belonging to $V^x_0(P)$,
\item[$(5)$] for every pair of vertices $u,v$ of $V^x_1(P)$ such that $u,v$ are 
     in different components of $\la N_G(x) \rag -R$, if the $(u,v)$-subpath 
     of $P$ contains an interior vertex, then it contains a vertex of 
     $N_G[x] \sm V^x_0(P)$ as an interior vertex.
\end{enumerate}
Moreover, if there is no $(a,b)$-path $P'$ in $G$ such that $V(P^*)=V(P')$, then
the following are satisfied:
\begin{enumerate} \setcounter{enumi}{5}
\item[$(6)$] for every $(a,b)$-paths $P,Q$ in $\Gstx$ such that 
     $V(P^*)=V(P)=V(Q)$, we have
     $a^x_P=a^x_Q$ and $b^x_P=b^x_Q$,
\item[$(7)$] $a^x_Pa$ and $b^x_Pb$ are edges of $G$.
\end{enumerate}
\end{corollary}

Fig.~\ref{f-vertices_on_path} shows a simple example of a possible position of 
the vertices $a^x_P$ and $b^x_P$ on $P$.

%
%
%
\begin{figure}[ht]
$$\bp
\setcoordinatesystem units <1mm,1mm>
\setplotarea x from -30 to 30, y from -7 to 7
\put{
\bp
\setcoordinatesystem units <1.0mm,1.0mm>
\setplotarea x from -20 to 20, y from -14 to 14
\put{$\bullet$} at    0    0 
\put{$\bullet$} at  -10   10 
\put{$\bullet$} at   10   10 
\put{$\bullet$} at   10  -10 
\put{$\bullet$} at  -10  -10 
\put{$\bullet$} at    0   18 
\put{$\bullet$} at    0  -18 
\put{$\bullet$} at  -17    5 
\put{$\bullet$} at  -17   -5 
\put{$\bullet$} at   17    5 
\put{$\bullet$} at   17   -5 
\setplotsymbol ({\large .})
\plot -10 10  0 18  10 10  17 5 /
\plot 17 -5  10 -10  0 -18  -10 -10  /
\plot -17 5  -17 -5  0 0   17 -5 /
\setplotsymbol ({\bf$\sim$})
\plotsymbolspacing=7.2pt
\plot -16.5 4.9  16.5 4.9 /
\plot -16.5 5.1  16.5 5.1 /
\plotsymbolspacing=0.4pt
\setplotsymbol ({\fiverm .})
\plot -17 5  -10 10 /
\plot -17 -5  -10 -10 /
\plot 17 5  17 -5 /
\plot 10 10  10 -10  -10 -10  -10 10  10 10 /
\plot 10 10  -10 -10 /
\plot 10 -10  -10 10 /
\plot -17 5  0 0  17 5 /
\plot -17 -5  0 0  17 -5 /
\plot 10 10  17 -5 /
\plot 10 -10  17 5 /
\plot -10 10  -17 -5 /
\plot -10 -10  -17 5 /
\ep} at  0 0
\put{$a$}     at   -12   12
\put{$a^x_P$} at    13   12
\put{$b^x_P$} at    13  -12
\put{$b$}     at   -12  -12
\put{$a^+$}   at     5   19
\put{$b^-$}   at     4  -19
\put{$x$}     at     1    3
\ep$$
\caption{An example of a possible position of the vertices $a^x_P$ and $b^x_P$ 
on $P$ (the path $P^*$ is in bold, the wavy line indicates a new edge).}
\label{f-vertices_on_path}
\end{figure}

\begin{proofbt} {\bf of Corollary~\ref{c-maximal}.}\quad
For simplicity, let $S$ denote the graph $\langle V(P^*)\rangle_{G}$.
Since clearly $x\in V(S)$ and $P^*$ is maximal, necessarily 
$N_G[x]\subset V(P^*)$, implying that $\la V(P^*)\ra_{\Gstx}=S^*_x$,
and $P^*$ is a longest path in~$S^*_x$.
%
%
%

Now, we prove the statement.
First, suppose that there is no $(a,b)$-path $P'$ in $G$ such that $V(P^*)=V(P')$.
Clearly, there is no such path $P'$ in $S$.
In~other words, every longest $(a,b)$-path in $S$ is shorter than $P^*$.
By Proposition~\ref{p-long}, $\langle N_S(x) \rangle_S$ is not $3$-connected,
which implies that its connectivity is $2$.
By Proposition~\ref{p-longest}, $\{a,b\}$ is a cut in $\langle N_S(x) \rangle_S$
(and in $\langle N_G(x) \rangle_G$).
Furthermore, properties $(1)$ -- $(7)$ are satisfied for 
every longest $(a,b)$-path in $S^*_x$,
that is, for every $(a,b)$-path $P$ such that $V(P^*)=V(P)$.
We observe that the properties are also satisfied in $\Gstx$.

Next, we suppose that $\langle N_G(x)\rangle_G$ is $2$-connected and
$\{a,b\}$ is a cut in $\langle N_G(x) \rangle_G$ (and in $\la N_S(x) \ra_S$),
and properties $(1)$ -- $(5)$ are satisfied
for every $(a,b)$-path $P$ in $\Gstx$ such that $V(P^*)=V(P)$.
We note that the properties are satisfied for every longest $(a,b)$-path in $S^*_x$.
By Proposition~\ref{p-longest},
every longest $(a,b)$-path in $S$ is shorter than $P^*$,
that is, there is no $(a,b)$-path in $S$ whose vertex set is $V(P^*)$,
and thus there is none in~$G$.
\end{proofbt}


\section{Proof of Theorem~\ref{t-closure}}
\label{s-proof}


In this section, we prove Theorem~\ref{t-closure}, that is,
a Tutte-closure of a claw-free graph
is a line graph of a hypergraph of rank at most $3$.

Let us give a brief outline of the proof.
Arguing with Corollary~\ref{c-maximal} and Lemma~\ref{l-non-adjacent},
we show that a Tutte-closure of a claw-free graph contains
at most one vertex whose neighborhood is non-complete and $2$-connected.
We remove this vertex, and we consider the resulting $2$-closed graph
and its clique covering given by Lemma~\ref{l-cover}.
We extend this covering to the whole graph, and we use
the case $r = 3$ of Theorem~\ref{t-hypergraphs}.

\bs

\begin{proofbt} {\bf of Theorem~\ref{t-closure}}. \quad
Let $G$ be a claw-free graph and let $G^T$ be its  Tutte-closure.
We note that $G^T$ is claw-free (by Lemma~\ref{l-free}).

We show that if $G^T$ is not connected or if it is $2$-closed,
then the statement is satisfied.
In the former case,
the components of $G^T$ are complete graphs (by definition).
In the latter case,
there is a system of cliques covering $E(G^T)$
such that every vertex of $G^T$ belongs to at most three cliques (by Lemma~\ref{l-cover}).
In both cases,
$G^T$ is a line graph of a hypergraph of rank at most $3$
(by Theorem~\ref{t-hypergraphs}).

We can assume that $G^T$ is connected and it is not $2$-closed.
By the definition of Tutte-closure,
there is a pair of vertices $a, b$ such that
no maximal $(a,b)$-path in $G^T$ is a Tutte path;
and we fix one such pair $a, b$.

We show the following.

%
%
\begin{claim}
\label{c-1}
Let $x$ be a vertex whose neighborhood in $G^T$ is non-complete and $2$-connected.
Then $\{a,b\}$ is a cut in $\langle N_{G^T}(x)\rangle_{G^T}$
(in particular, $x$ is adjacent to $a$ and $b$). 
\end{claim}

\begin{proofcl}
We consider such a vertex $x$. 
By definition, the graph $(G^T)^*_x$ has a maximal $(a,b)$-path, say $P$, 
which is a Tutte path,
and there is no $(a,b)$-path $P'$ in $G^T$ such that $V(P) = V(P')$.  
By Corollary~\ref{c-maximal},
$\{a,b\}$ is a cut in $\langle N_{G^T}(x)\rangle_{G^T}$.
\end{proofcl}

We consider a vertex $x$ whose neighborhood in $G^T$ is non-complete and 
$2$-connected, and a maximal $(a,b)$-path $P$ in $(G^T)^*_x$ which is a 
Tutte path, and we make a special choice of $P$ in $(G^T)^*_x$ as follows.
We consider the vertices $a^x_P$ and $b^x_P$ given by Corollary~\ref{c-maximal}
(applied to $G^T$).
In particular, the choice implies that set $V^x_1(P)$ is non-empty;
and so there is an edge, say $e$, of $P$
whose both ends belong to $N_{G^T}[x]$,
and we use this fact and show that we can choose $P$ such that no
interior vertex of $P$ belongs to $V^x_0(P)$.
If there is such a vertex, say $i$,
then we consider the graph $\langle N_{P}(i) \cup \{ i,x\} \rangle_{G^T}$,
and we note that the vertices of $N_{P}(i)$ are joined by an edge (since $G^T$ is 
claw-free).
We modify $P$ by adding this edge and removing edges incident with $i$,
and by adding the edges joining $i$ to the vertices incident with $e$ and 
removing $e$. The choice of $P$ follows.
(Note that, by the definition of the sets $V_i^x(P)$, vertices in $V_0(P)$ 
cannot be consecutive on $P$, and hence the elimination of such a vertex 
does not affect a possibility to eliminate others).

We fix this choice of $x$ and $P$, and we simplify the notation by letting
$a'$ denote $a^x_P$, and $b'$ denote $b^x_P$
(we recall that these are the first and the last vertex on $P$ belonging to $V^x_1(P)$).
We let $a^+$ denote the neighbor of $a$ on $P$.

We observe that the graph $\langle N_{G^T}(x)\rangle_{G^T} - \{a,b\}$
has precisely two components
(since $G^T$ is claw-free, $\langle N_{G^T}(x)\rangle_{G^T}$ is $2$-connected,
and $\{a,b\}$ is a cut in $\langle N_{G^T}(x)\rangle_{G^T}$ by Claim~\ref{c-1}).
Properties (2) and (3) of Corollary~\ref{c-maximal} imply that
$a'$ and $b'$ belong to the same component of $\langle N_{G^T}(x)\rangle_{G^T} - \{a,b\}$;
and we let $C'$ denote the set of all vertices of this component,
and $C$ denote the set of all vertices of the other component.

We show the following.

%
%
\begin{claim}
\label{c-2}
No vertex of $N_{G^T}(x)\setminus \{a,b,a'\}$ is adjacent to $a^+$.
Furthermore, $N_{G^T}(a)\cap C'=\{a'\}$ and $N_{G^T}(b)\cap C'=\{b'\}$
(in particular, $ab'$ and $a'b$ are non-edges in~$G^T$).
\end{claim}

\begin{proofcl} 
Suppose, to the contrary, that there is such a vertex, say, $w$.
We recall that $P$ contains all vertices of $N_{G^T}[x]$,
and no interior vertex of $P$ belongs to $V^x_0(P)$
(by the choice of $P$).
In~particular, we can choose an edge of $P$ 
incident with two vertices of $N_{G^T}[x]$ one of which is $a'$,
and analogously choose an edge of $P$ incident with 
 two vertices of $N_{G^T}[x]$ one of which is $w$
(it might be the same edge).  
We consider subpaths of $P$ obtained by removing these edges (this edge)
and the edge $a a^+$,
and we connect two of these subpaths using the edge $a^+w$.
We observe that each of the obtained subpaths has both ends in $N_{G^T}[x]$, 
and so we can join these subpaths into an $(a,b)$-path 
(since $N_{(G^T)^*_x}[x]$ induces a clique).
Let $P'$ denote the resulting path.
Clearly, $P'$ is a maximal $(a,b)$-path (since $V(P) = V(P')$). We note 
that $a = a^x_{P'}$, contradicting property (2) of Corollary~\ref{c-maximal}.

We show that $N_{G^T}(a)\cap C'=\{a'\}$
(the argument for showing $N_{G^T}(b)\cap C'=\{b'\}$ is similar).
By item (7) of Corollary~\ref{c-maximal}, $a$ is adjacent to $a'$,
thus $a'$ belongs to $N_{G^T}(a)\cap C'$.
For every vertex $u$ of $C' \setminus \{a'\}$,
we show that it is not adjacent to $a$.
We consider a vertex, say $v$, of $N_{G^T}(a)\cap C$
and the graph $\langle \{a,a^+,u,v\} \rangle_{G^T}$.  
First, we observe that vertices $a,a^+,u,v$ are distinct
(property (2) of Corollary~\ref{c-maximal} implies that $a$ belongs to $V^x_0(P)$,
that is, $a^+$ is not adjacent to $x$, and thus it is distinct from $u$ and $v$).
Next, we discuss the edges.
By definition, $a$ is adjacent to $a^+$ and $v$, and $u$ is not adjacent to $v$
(since they are in different components of
$\langle N_{G^T}(x)\rangle_{G^T} - \{a,b\}$).
By the first part of Claim~\ref{c-2}, $a^+$ is adjacent to neither $u$ nor $v$. 
We conclude that $u$ is not adjacent to $a$ (since $G^T$ is claw-free).
\end{proofcl}

We let $A$ denote the set $N_{G^T}(a) \setminus N_{G^T}[x]$,
and $B$ denote $N_{G^T}(b) \setminus N_{G^T}[x]$.
We study sets $A$, $B$, $C$ and $C'$ (see Figure~\ref{f-abc})
in the following claims.

%
%
%
\begin{figure}[ht]
$$\bp
\setcoordinatesystem units <1mm,1mm>
\setplotarea x from -30 to 30, y from -25 to 25
\put{
\bp
\setcoordinatesystem units <0.8mm,0.8mm>
\setplotarea x from -20 to 20, y from -14 to 14
\put{$\bullet$} at    0    0 
\put{$x$} at    5   2 
\put{$\bullet$} at  -20    0 
\put{$a$} at   -22   2.5 
\put{$\bullet$} at  -38   -7 
\put{$a^+$} at   -38   -3 
\put{$\bullet$} at   20    0 
\put{$b$} at    22   2.5 
\put{$\bullet$} at  -20  -20 
\put{$a'$} at   -22  -22.5 
\put{$\bullet$} at   20  -20 
\put{$b'$} at    22  -22.5 
\circulararc  360  degrees from  -25 -10   center at -35  -10
\put{$A$} at  -35   -10 
\circulararc  360  degrees from   25 -10   center at  35  -10
\put{$B$} at  35   -10 
\ellipticalarc axes ratio 2:1 360 degrees from -20 -20 center at 0 -20
\put{$C'$} at  0  -20 
\ellipticalarc axes ratio 2:1 360 degrees from -20  20 center at 0  20
\put{$C$} at  0   20 
\plot -35 0  35 0 /
\plot -20 -20  -20 20 /
\plot  20 -20   20 20 /
\plot 0  0   -17  14.7 /
\plot 0  0    17  14.7 /
\plot 0  0    10  15   /
\plot 0  0   -10  15   /
\plot 0  0     0  16   /
\plot 0  0   -17 -14.7 /
\plot 0  0    17 -14.7 /
\plot 0  0    10 -15   /
\plot 0  0   -10 -15   /
\plot 0  0     0 -16   /
\plot -20 0  11 11.7 /
\plot -20 0  -16 18  /
\plot -20 0   -7 16  /
\plot -20 0    3 14  /
\plot  20 0  -11 11.7 /
\plot  20 0   16 18  /
\plot  20 0    7 16  /
\plot  20 0   -3 14  /
\plot -20 0   -25.8 -14  /
\plot -20 0   -29   -9   /
\plot -20 0   -32   -5   /
\plot  20 0    25.8 -14  /
\plot  20 0    29   -9   /
\plot  20 0    32   -5   /
\plot -20 -20   -25.8  -6  /
\plot -20 -20   -29  -11   /
\plot -20 -20   -32   -15  /
\plot -20 -20   -35   -20  /
\plot  20 -20    25.8  -6  /
\plot  20 -20    29  -11   /
\plot  20 -20    32   -15  /
\plot  20 -20    35   -20  /
\ep} at  0 0
\ep$$
\bsm
\caption{The sets $A$, $B$, $C$, $C'$ are pairwise disjoint and
    each of the sets $A \cup \{a, a'\}$, $B \cup \{b, b'\}$,
    $C \cup \{a, x\}$, $C \cup \{b, x\}$, $C' \cup \{x\}$
    induces a clique.}
\label{f-abc}
\end{figure}

%
%
\begin{claim}
\label{c-3}
Each of the sets $A$, $B$, $C \cup \{a\}$, $C \cup \{b\}$ and $C'$ induces a clique in~$G^T$.
\end{claim}

\begin{proofcl} 
Suppose, to the contrary, that some of these sets contains a pair $u,v$ of 
non-adjacent vertices.
If $u,v\in A$, then $\lab a,u,v,x\rab_{G^T}$ is a claw, a contradiction. 
Similarly, if $u,v\in B$, then $\lab b,u,v,x\rab_{G^T}$ is a claw,
if $u,v\in C\cup\{a\}$, then $\lab x,u,v,b'\rab_{G^T}$ is a claw,
if $u,v\in C\cup\{b\}$, then $\lab x,u,v,a'\rab_{G^T}$ is a claw, and 
if $u,v\in C'$, then $\lab x,u,v,w\rab_{G^T}$ for $w\in C$ is a claw.
\end{proofcl}

%
%
\begin{claim}
\label{c-4}
No vertex of $A \cup B$ is adjacent to a vertex of $C$.
Furthermore, both $A \cup \{a, a'\}$ and $B \cup \{b, b'\}$ induce a clique 
in $G^T$.
\end{claim}

\begin{proofcl}
We use Claims~\ref{c-2} and~\ref{c-3}
and give the proof for $A$ (the arguments for $B$ are similar).

For the sake of a contradiction, suppose that there is a vertex, say, $v$, 
of $C$ which is adjacent to a vertex $u\in A$. 
We note that $a^+$ is not adjacent to $v$
(in particular, $a^+$ is distinct from $u$).
Considering the graph $\langle \{a,a^+,a',v\} \rangle_{G^T}$,
we see that $a^+$ is adjacent to $a'$
(since $G^T$ is claw-free).
We observe that $\langle \{v,u,a^+,a',x\} \rangle_{G^T}$
is a $2$-connected subgraph of $\langle N_{G^T}(a)\rangle_{G^T}$
and $u,x$ and $a',v$ are two pairs of non-adjacent vertices.
By Lemma~\ref{l-non-adjacent}, $\langle N_{G^T}(a)\rangle_{G^T}$ is $2$-connected,
and we obtain a contradiction with Claim~\ref{c-1}. 

Furthermore, we recall that
$a$ is adjacent to $a'$ (by Claim~\ref{c-2})
and $A$ induces a clique (by Claim~\ref{c-3}).
Clearly, $u$ is adjacent to $a$.
Considering $\langle \{a,a',u,v\} \rangle_{G^T}$
(where $v$ is a vertex of $C$), we see that $u$ is adjacent to $a'$
(since $G^T$ is claw-free).
Thus, the set $A \cup \{a, a'\}$ induces a clique.
\end{proofcl}

%
%
\begin{claim}
\label{c-5}
$N_{G^T}(a) \cap N_{G^T}(b) = C  \cup \{x\}$.
\end{claim}

\begin{proofcl} 
We note that every vertex of $C  \cup \{x\}$ belongs to 
$N_{G^T}(a) \cap N_{G^T}(b)$ (by Claim~\ref{c-3}).

Suppose, to the contrary, that there is a vertex $z$ of
$N_{G^T}(a) \cap N_{G^T}(b)$ not belonging to $C  \cup \{x\}$.
By Claim~\ref{c-2}, $z$ does not belong to $C'$.
Consequently, $z$ does not belong to $N_{G^T}[x]$,
hence it belongs to $A \cap B$, and thus it is adjacent to $a'$ and $b'$
(by Claim~\ref{c-4}).
Using Claims~\ref{c-2}, \ref{c-3} and~\ref{c-4},
we observe that the graph $\langle \{a, z, b', x\} \rangle_{G^T}$
is a $2$-connected subgraph of
$\langle N_{G^T}(a')\rangle_{G^T}$
and $a, b'$ and $z, x$ are two pairs of non-adjacent vertices.
By Lemma~\ref{l-non-adjacent}, $\langle N_{G^T}(a') \rangle_{G^T}$ is $2$-connected.
By Claim~\ref{c-1}, $a'$ is adjacent to $b$, which contradicts Claim~\ref{c-2}.
\end{proofcl}

Now we show that $x$ is the only vertex whose neighborhood in $G^T$
is non-complete and $2$-connected.
Suppose, to the contrary, that there is such a vertex $y$ distinct from~$x$.
By Claim~\ref{c-1}, $y$ is adjacent to $a$ and $b$,
and thus it belongs to $C$ by Claim~\ref{c-5}.
In particular, $x$ and $y$ are adjacent.
Furthermore, $\{a,b\}$ is a cut in $\langle N_{G^T}(y) \rangle_{G^T}$.
We consider a component of
$\langle N_{G^T}(y) \rangle_{G^T} - \{a,b\}$ not containing $x$
(clearly, this component contains no vertex of $N_{G^T}[x]$);
and we choose a vertex, say $z$, of this component
such that $z$ is adjacent to $a$. 
Thus, $z$ belongs to $A$.
By Claim~\ref{c-4}, we get that $z$ is not adjacent to $y$,
a contradiction. 

Since the neighborhood of every vertex except of $x$
is complete or not $2$-connected, the graph $G^T - x$ is $2$-closed,
hence Lemma~\ref{l-cover} gives a suitable clique covering $\mathcal{K}_0$
of $G^T - x$.

In addition, we define a system $\mathcal{K}_{x}$ of cliques covering all edges 
incident with $x$.
We discuss two cases.
In case $a$ and $b$ are not adjacent, we recall that each of the sets 
$C \cup \{a,x\}$, $C \cup \{b,x\}$ and $C'\cup\{x\}$ induces a clique in~$G^T$
(by Claim~\ref{c-3}), and
we let $\mathcal{K}_{x}$ be the system consisting of these three cliques.
Otherwise, we note that $C \cup \{a,b,x\}$ induces a clique and we let
$\mathcal{K}_{x}=\{\la C\cup\{a,b,x\}\ra_{G^T},\la C'\cup\{x\}\ra_{G^T}\}$.

We show that the cliques of $\mathcal{K}_{x}$ are maximal cliques in $G^T-x$.
Consider a set $S$ of vertices inducing a clique of $\mathcal{K}_{x}$,
and suppose, to the contrary, that there is a vertex, say $u$, such that 
$S \cup \{u\}$ induces a clique in $G^T-x$.
We discuss two cases.
In case $a$ or $b$ belongs to $S$, we observe that $u$ belongs to 
$A \cup B \cup C'$; but Claim~\ref{c-4} implies that $u$ is not adjacent to 
any vertex of $C$, a contradiction.
Otherwise, we have that $S = C'$.
By definition, $u$ is adjacent to $a'$ and $b'$.
In~particular, $a^+$ is distinct from $u$ (by Claim~\ref{c-2}).
Considering $\langle \{a',x,a^+,u\} \rangle_{G^T}$, we have that $a^+$ is 
adjacent to $u$ (since $G^T$ is claw-free).
We observe that $\langle \{a,x,b',u,a^+\} \rangle_{G^T}$
is a $2$-connected subgraph of $\langle N_{G^T}(a')\rangle_{G^T}$
and $a,b'$ and $a^+,x$ are two pairs of non-adjacent vertices.
By Lemma~\ref{l-non-adjacent}, $\langle N_{G^T}(a')\rangle_{G^T}$ is 
$2$-connected, a contradiction.

We consider the (already defined) system $\mathcal{K}_{0}$ of cliques covering
$G^T - x$, and we extend the system to $G^T$.
For every clique $K$ of $\mathcal{K}_{x}$, we discuss two cases.
In case $K$ belongs to $\mathcal{K}_0$, we replace $K$ by the clique
$\langle V(K) \cup \{x\} \rangle_{G^T}$ in the system.
Otherwise, we observe that no vertex of $K$ belongs to three cliques of 
$\mathcal{K}_0$ (by Lemma~\ref{l-cover} since $K$ is a maximal clique in $G^T-x$).
If $K$ is induced by $C'$ or $C \cup \{a\}$ or $C \cup \{a, b\}$, then we add
$\langle V(K) \cup \{x\} \rangle_{G^T}$ to the system.
If $K$ is induced by $C \cup \{b\}$, then we add $\langle \{b, x\} \rangle_{G^T}$.

The resulting system of cliques covers $E(G^T)$ such that every vertex of $G^T$
belongs to at most three cliques.
We conclude that $G^T$ is a line graph of a hypergraph of rank at most $3$
(by Theorem~\ref{t-hypergraphs}).
\end{proofbt}


\section{Concluding remark}
\label{sec-concluding}

We remark that the Tutte-closure of a graph can be alternatively defined 
with an additional condition that the local completion operation is only 
applied, among vertices $x$ for which $\Gstx$ remains non-Tutte-connected,
to those having 2-connected neighborhood.
From the proof of Theorem~\ref{t-closure} one can see that using this alternative
definition does not affect the reasoning throughout the paper.


\end{document}